\documentclass[a4paper]{article}

\usepackage{indentfirst}
\usepackage{amsmath}
\usepackage{amssymb}
\usepackage{latexsym}
\usepackage{ntheorem}

\usepackage{ifthen}

\theoremstyle{plain}

\theoremheaderfont{\bf}\theorembodyfont{\sl}

\newtheorem{theorem}{Theorem}[section]
\newtheorem{proposition}[theorem]{Proposition}
\newtheorem{lemma}[theorem]{Lemma}

\theoremheaderfont{\it}\theorembodyfont{\rm}

\newtheorem{remark}[theorem]{Remark}

\theoremheaderfont{\bf}\theorembodyfont{\rm}

\newtheorem{definition}[theorem]{Definition}
\newtheorem{example}[theorem]{Example}

\newtheorem{assumption}[theorem]{Assumption}

\theoremstyle{nonumberplain}

\theoremheaderfont{\bf}\theorembodyfont{\sl}

\newenvironment{proof}[1][]
{\ifthenelse{\equal{#1}{}}{\smallskip\noindent\textsl{Proof. }}{\smallskip
\noindent\textsl{Proof #1. }}}{\hfill$\Box$}

\newcommand{\conv}{\underset{n\rightarrow +\infty}{\longrightarrow}}

\begin{document}

\title{Risk-averse asymptotics for reservation prices}


\author{Laurence Carassus \thanks{Laboratoire de Probabilit\'es et Mod\`eles Al\'eatoires,
Universit\'e Paris 7 Denis Diderot, 16 rue Clisson, 75013 Paris,
France (e-mail: carassus@math.jussieu.fr)} \and Mikl\'os
R\'asonyi\thanks{The Computer and Automation Institute of the
Hungarian Academy of Sciences, 1518 Budapest, P. O. Box 63., Hungary
(e-mail: rasonyi@sztaki.hu). The second author acknowledges
financial support from the Hungarian Science Foundation (OTKA) grant
F 049094 and dedicates this paper to Annam\'aria Brecz.}}

\date{\today}

\maketitle

\begin{abstract}

An investor's risk aversion is assumed to tend to infinity. In a fairly general setting, we present
conditions ensuring that the respective utility indifference prices of a given contingent
claim converge to its superreplication price.


\end{abstract}

\noindent \textbf{Keywords:} Utility indifference price,
Superreplication price, Convergence, Utility maximization, Risk
aversion.

\noindent \textbf{MSC 2000 Subject Classification:} Primary: 91B16,
91B28 ; Secondary: 93E20, 49L20

\noindent \textbf{OR/MS Subject Classification:} Primary: Utility /
Value theory ; Secondary: Finance / Asset Pricing

\noindent \textbf{JEL classification:} C61, C62, G11, G12

\section{Introduction}

In this article we investigate the effect of increasing risk aversion on utility-based prices.
We are dealing with the utility indifference price (or reservation price), defined in \cite{hodges}
for the first time. This is the minimal amount added to an option seller's initial capital
which allows her to attain the same utility that she would have attained from her initial capital
without selling the option, see Definition \ref{kas} below. Intuitively, when risk aversion tends
to infinity, reservation price should tend
to the superreplication price (i.e. the price of hedging the option without any risk).

This result was shown in \cite{rouge} for Brownian models and in \cite{sixauthor} in a semimartingale setting when the agent has constant absolute
risk aversion (i.e. for exponential utility functions). Certain other classes of utility functions were treated in \cite{bouchard-these}, models with transaction costs
were considered in \cite{bouchard}.

However, an extension of this result to general utility functions was lacking. In \cite{laurence1}
and \cite{laurence2} the case of discrete-time markets was treated for utilities on the
positive axis as well as on the real line. Now we prove this result in a continuous-time
semimartingale framework, under suitable hypotheses.

In section \ref{s1} we model the agent's preferences and
introduce a growth condition (related to the elasticity of utility functions), in
section \ref{s2} the market model and
a compactness assumption are discussed. In section \ref{s3} the concept of utility indifference
price is formally defined and the two main theorems are proved.

\section{Risk averse agents}\label{s1}

We consider agents trading in the market with initial endowment $z\in\mathbb{R}$. We assume that

\begin{assumption}\label{a1}
$U_n$, $n\in\mathbb{N}$ are twice continuously differentiable,
strictly concave and increasing functions on $\mathbb{R}$ such that for each $x\in\mathbb{R}$,
\begin{equation}\label{r}
r_n(x):=\frac{-U_n''(x)}{U_n'(x)} \conv \infty.
\end{equation}
\end{assumption}

The function $r_n$ is called the (absolute)
risk aversion of an agent with utility function $U_n$. This concept was introduced in
\cite{arrow} and \cite{pratt}. Here we are interested in the
case where this measure of risk aversion tends to infinity.

We take the Fenchel conjugates of ${U}_n$:
\begin{equation}
{V}_n(y):=\sup_{x\in\mathbb{R}}\{ U_n(x)-xy\},\quad y\in (0,\infty).
\end{equation}
As easily checked, ${V}_n$ is a finite convex function.

We stipulate a growth condition on the conjugates of the utility functions we consider. Such assumptions are often
referred to as ``elasticity conditions''.

\begin{assumption}\label{elastic} For each $[\lambda_0,\lambda_1]\subset (0,\infty)$ there exist positive constants
$C_1,C_2,C_3$ such that for all $n$ and for all $y>0$,
\begin{equation}\label{GRRR}
{V}_n(\lambda y)\leq C_1{V}_n(y)+C_2y+C_3.
\end{equation}
holds for each $\lambda\in [\lambda_0,\lambda_1]$.
\end{assumption}

\begin{remark} For $n$ fixed, condition \eqref{GRRR} is equivalent to
\begin{equation}\label{okt}
\limsup_{x\to\infty}\frac{x{U}_n'(x)}{{U}_n(x)}<1,\quad
\liminf_{x\to -\infty}\frac{x{U}_n'(x)}{{U}_n(x)}>1,
\end{equation}
as shown in \cite{frittelli}.


The first of the two conditions in \eqref{okt} was introduced in
\cite{kramkov-schachermayer}, the second one in
\cite{schachermayer}. A utility function ${U}_n$ satisfying
\eqref{okt} is said to have \emph{reasonable asymptotic elasticity}
(terminology of \cite{schachermayer}). Thus condition \eqref{GRRR} is a
(dually formulated) uniform reasonable asymptotic elasticity
condition.
Another uniform asymptotic elasticity condition appears as
Assumption 2.3 of \cite{laurence3}. About the derivation of
equivalances like that of \eqref{okt} and \eqref{GRRR} consult section
6 of \cite{kramkov-schachermayer}, section 4 of \cite{schachermayer}
and \cite{frittelli}.
\end{remark}



\section{Market Model}\label{s2}

Our market is modelled by an adapted $d$-dimensional stochastic process $S$ on a given continuous-time
stochastic basis $(\Omega,\mathcal{F},(\mathcal{F}_t)_{t\geq 0},P)$. We think that $S$
represents the  evolution of the (discounted) prices of $d$ assets. For simplicity we assume $S$ locally
bounded (to avoid technical complications related to $\sigma$-martingales). Absence of arbitrage is stipulated by
\[
\mathcal{M}\neq \emptyset,
\]
where
$\mathcal{M}$ denotes the set of measures $Q\sim P$ such that $S$ is a local martingale
under $Q$.

We make the following compactness assumption.

\begin{assumption}\label{compact}
There exists $Q_0\in\mathcal{M}$ such that the sequence
${V}_n(dQ_0/dP)$ is uniformly integrable (with respect to $P$). We
denote by $\mathcal{M}_v$ the set of such $Q_0$s.
\end{assumption}

\begin{remark}
Similar conditions have already appeared in investigations on the stability of optimal strategies
with respect to perturbations of utility functions, see \cite{larsen} and \cite{costas}.
Both papers consider a sequence $U_n$ of utility functions converging to a limiting utility $U$
and show convergence of the corresponding optimal strategies/utility prices.

In \cite{larsen} a sequence $U_n$ is
assumed to be dominated by some $\overline{U}$ with conjugate function $\overline{V}$ such that
$\overline{V}(Z)$ is integrable, where $Z$ is (the candidate for) the density of
the minimal martingale measure. In Assumption (UI) of \cite{costas} we find
essentially the uniform integrability of $V_n(y dQ_n/dP)$ for each $y>0$, where
$Q_n$ is a suitable sequence in $\mathcal{M}$. Note that \cite{costas}
also considers perturbations of the underlying probability measure and that in both mentioned
papers the $U_n$ are defined on the positive axis.
\end{remark}

This elasticity hypothesis \ref{elastic} allows us to prove that  the set $\mathcal{M}_v$ is in fact large.
\begin{lemma}\label{dense} Assume that there is $x_0\in\mathbb{R}$ such that
$U_n(x_0)$ is bounded from below. Under Assumptions \ref{elastic}
and \ref{compact}, the set $\mathcal{M}_v$ is dense in $\mathcal{M}$
with respect to the total variation norm topology.
\end{lemma}
\begin{proof} This result was essentially reported in Proposition 6
of \cite{unbounded}, without proof.
Fix $Q_0\in\mathcal{M}_v$. We know from \cite{kabanov-stricker} that the set
\[
\mathcal{M}_b:=\{Q\in\mathcal{M}: dQ/dQ_0\mbox{ is bounded }\}
\]
is dense in $\mathcal{M}$. It follows by concavity of $\mathcal{M}$ that also
\[
\mathcal{M}_{bb}:=\left\{ \alpha Q_0+\left(1-\alpha \right)Q:\ 0<\alpha< 1,\ Q\in\mathcal{M}_b\right\}
\]
is a dense subset of $\mathcal{M}$. We shall show that
$\mathcal{M}_{bb}\subset\mathcal{M}_{v}$. Take an arbitrary
$Q\in\mathcal{M}_b$ such that $dQ/dQ_0\leq K$ and any fixed $0<\alpha<1$. We have
\begin{eqnarray*}
{V}_n(\alpha dQ_0/dP +(1-\alpha)dQ/dP)& \geq & U_n(x_0)-x_0 (\alpha
dQ_0/dP+(1-\alpha)dQ/dP).
\end{eqnarray*}
Moreover,
\begin{eqnarray*}
{V}_n(\alpha dQ_0/dP +(1-\alpha)dQ/dP) & \leq & C_1{V}_n(dQ_0/dP) +C_2(dQ_0/dP)+C_3,
\end{eqnarray*}
by Assumption \ref{elastic} (choose $C_1,C_2,C_3$ corresponding to
the interval $[\alpha, \alpha+(1-\alpha)K]$ and note that
$\alpha+(1-\alpha)dQ_0/dQ$ falls into this interval).

Uniform integrability of ${V}_n(\alpha
dQ_0/dP+(1-\alpha)dQ/dP),n\in\mathbb{N}$ now follows from Assumption
\ref{compact}.
\end{proof}

\section{Utility indifference prices}\label{s3}

It follows from $\mathcal{M}\neq\emptyset$
that $S$ is a semimartingale and we may define the set of admissible trading strategies as
predictable $S$-integrable processes with a finite credit line (to avoid
doubling strategies).

\begin{definition}
\[
\mathcal{A}:=\{\phi:\phi\mbox{ predictable},S\mbox{-integrable and for some }w>0,
V^{0,\phi}\geq -w\},
\]
where we write $V^{x,\phi}_t:=x+(\phi\cdot S)_t$ for the value process of strategy $\phi$
starting from initial endowment $x$.
\end{definition}

Fix $T>0$ and a bounded random variable $G$, interpreted as a contingent claim to be delivered at the
end of the period $[0,T]$.
By optimal trading an agent with initial capital $x\in\mathbb{R}$ and utility $U_n$ delivering
the claim $G$ may attain
\begin{equation}
\label{max} u_n(x,G):=\sup_{\phi\in\mathcal{A}}
EU_n(V^{x,\phi}_T-G),
\end{equation}
this is well-defined as admissible strategies have bounded from
below value processes.

\begin{definition}\label{kas}
The utility indifference price for $U_n$ and initial endowment $z$ is
\begin{equation}\label{xx}
p_n(z,G):=\inf\{ p\in\mathbb{R}: u_n(z+p,G)\geq u_n(z,0)\},
\end{equation}
the minimal extra capital that allows for delivering $G$ while attaining the same utility as
without claim delivery.
\end{definition}

The utility-free concept of the superhedging price of a bounded contingent claim (random variable) $G$
is defined as
\[
\pi(G):=\inf\{x\in\mathbb{R}:\mbox{ there is }\phi\in\mathcal{A}\mbox{ such that }
V_T^{x,\phi}\geq G\mbox{ a.s.}\}.
\]

The following dual characterization is a fundamental result of mathematical finance, see e.g. \cite{follmer-kabanov}
and the references therein.
\begin{theorem}\label{superhedging}
Under $\mathcal{M}\neq\emptyset$, we have
\[
\pi(G)=\sup_{Q\in\mathcal{M}} E_QG.
\]
\end{theorem}

We are now ready to state our main result i.e. the convergence of
utility indifference prices to the superreplication price. We will
present this result under two types of assumptions. The first one
(see Assumption \ref{a2}) refers to the existence of a initial
wealth for which all the investors have (asymptotically) a common
preference for $x_0$ and also a common non zero growth rate for
their preferences near $x_0$.

The second kind of assumptions impose the elasticity Assumption \ref{elastic} and
the compactness Assumption \ref{compact} on a normalized family of utility functions (see Theorem \ref{main2}).

\begin{assumption}\label{a2}
There exists some $x_0,\alpha, \beta \in \mathbb{R}$  with $\alpha
\in(0,\infty)$ and $|\beta|<\infty$ such that :
\begin{eqnarray*}
U_n'(x_0) & \conv & \alpha \\
U_n(x_0) & \conv & \beta.
\end{eqnarray*}
\end{assumption}

\begin{remark}
We will see in the examples below that this assumption is satisfied
for power and exponential utility functions. In those examples, we
have that $U_n(0)=0$ and $U_n'(0)=1$. This means that all investors
consider that utility of nothing is zero and that for very small
wealth their utility functions are approximatively linear with slope
one.
\end{remark}

The following proposition is the first step to Theorem \ref{main} :
\begin{proposition}\label{nulla}
Under Assumptions \ref{a1} and
\ref{a2}, for each $y>0$,
\[
{V}_n(y)\conv \beta -x_0y.
\]
\end{proposition}
\begin{proof}
First, the argument of Lemma 4 in \cite{laurence1} shows that as $n\to\infty$,
\begin{eqnarray}
\label{c}
{U}_n'(x) \to \infty,\ x<x_0, & \quad & {U}_n'(x)\to 0,\ x>x_0.
\end{eqnarray}
Let $I_n$ be the inverse of ${U}_n'$ (which exists by strict
concavity of ${U}_n$).  We claim
\[
I_n(y)\to x_0,\mbox{ for }y>0.
\]
Indeed, let $y>0$. If we had $I_{n_k}(y) \geq x_0 +\varepsilon$ for some
$\varepsilon>0$ and a subsequence $n_k$ then
$U_{n_k}'(I_{n_k}(y))=y\leq {U}_{n_k}'(x_0+\varepsilon)$, but this
latter tends to $0$ by (\ref{c}), a contradiction.
The proof of the reverse inequality is similar and $I_n(y) \to x_0$.

First by definition ${V}_n(y)\geq U_n(x_0) -x_0y$. Since ${V}_n$ can
be calculated as ${V}_n(y)={U}_n(I_n(y))-I_n(y)y$; concavity of
${U}_n$ implies that ${V}_n(y)\leq U_n(x_0) +(I_n(y)-x_0) U'_n(x_0)
-I_n(y)y$. Thus
$$0\leq V_n(y)- (U_n(x_0) -x_0y) \leq (I_n(y)-x_0) (U'_n(x_0)
-y)\to 0,$$ $n\to\infty$, showing the claim.
\end{proof}





\begin{theorem}\label{main} If Assumptions \ref{a1},  \ref{elastic}, \ref{compact} and \ref{a2} hold then
for each bounded random variable $G$ the quantities $ u_n(x_0,G),
n\in\mathbb{N}$ are finite and the corresponding utility
indifference prices $p_n(x_0,G)$ tend to $\pi(G)$ as $n\to\infty$.
\end{theorem}
\begin{remark}
The convergence in Theorem \ref{main} holds only for $x_0$. We will see in Theorem \ref{main2} how
the convergence for all initial capital $z$ can be checked under a second type of assumptions.
\end{remark}
\begin{proof}
It is a standard fact that that $p_n(x_0,G)\leq\pi(G)$, see page 152
of \cite{laurence1}. For the reverse inequality, we argue by
contradiction. Suppose that for some $\varepsilon>0$ and a
subsequence $n_k$ one has $p_k:=p_{n_k}(x_0,G)\leq
\pi(G)-\varepsilon$ for each $k$. We may suppose that $n_k=k$ and
$p_k\to \pi(G)-\varepsilon$, $k\to\infty$.

Take arbitrary $\phi\in\mathcal{A}$ and $Q\in\mathcal{M}_v$. Then
$E_Q(V^{x_0+p_k,\phi}_T-G)\leq x_0+p_k-E_Q G$ ($V^{0,\phi}$ is a
supermartingale by results of \cite{ansel}). Proposition \ref{nulla}
and $Q\in\mathcal{M}_v$ imply
\begin{eqnarray}\label{x}
\limsup_{k\to\infty}E({V}_k(dQ/dP)+(dQ/dP)(V^{x_0+p_k,\phi}_T-G)) \\
\nonumber \leq \beta-x_0+ (x_0+\limsup_{k\to\infty}p_k-E_Q G).
\end{eqnarray}
Introduce
\[
v_k^{\phi}:=\inf_{Q\in\mathcal{M}_v}
E({V}_k(dQ/dP)+(dQ/dP)(V^{x_0+p_k,\phi}_T-G)).
\]

It follows from Lemma \ref{dense} that $\sup_{Q\in\mathcal{M}_v}E_Q G=
\sup_{Q\in\mathcal{M}}E_Q G$, thus \eqref{x} and Theorem \ref{superhedging} imply
\[
\limsup_{k\to\infty} v_k^{\phi}\leq \beta +
\lim_{k\to\infty}p_k-\pi(G)=\beta-\varepsilon.
\]
By the definition of conjugate functions we have,
\[
\limsup_{k\to\infty}E{U}_k(V_T^{x_0+p_k,\phi}-G)\leq
\limsup_{k\to\infty} v_k^{\phi}\leq \beta -\varepsilon.
\]
It follows that ${u}_k$ are finite and
$\limsup_{k\to\infty}{u}_k(x_0+p_k,G)\leq \beta -\varepsilon$. But ${u}_k(x_0+p_k,G)\geq{u}_k(x_0,0)\geq
{U}_k(x_0)$, thus $\limsup_{k\to\infty}{u}_k(x_0+p_k,G) \geq \beta$
a contradiction.
\end{proof}
\begin{remark} Let us consider the condition that there
exists $Q_0\in\mathcal{M}$ such that $dQ_0/dP$ and $dP/dQ_0$ are both bounded (by some $K>0$).

Note that a convex
function attains its maximum on an interval at one of the endpoints. Hence
\[
U_n(x_0)-\vert x_0\vert K\leq {V}_n(dQ_0/dP)\leq \vert
{V}_n(K)\vert+\vert{V}_n(1/K)\vert
\]
which is bounded
by Proposition \ref{nulla}, showing that this condition is stronger that Assumption \ref{compact}.

We claim that, replacing Assumption \ref{compact} by this condition, one may drop Assumption \ref{elastic}
from the hypotheses of Theorem \ref{main}. Indeed, the result of
\cite{kabanov-stricker} directly implies that measures $Q_0$ with the above property are dense in $\mathcal{M}$,
without appeal to Assumption \ref{elastic}. The rest of the proof is identical.

Unfortunately, in continuous-time models one rarely finds such a $Q_0$.
In discrete-time models, however, such measures often exist, see \cite{rokhlin} for an extensive
discussion.
\end{remark}

We now turn to our second Theorem. To state its hypothesis, we first
need to introduce some normalization of the functions $U_n$. Fix some initial wealth $z\in\mathbb{R}$ and
set~:
\begin{eqnarray}
\label{tilde}
\tilde{U}_n(x):=\frac{U_n(x)-U_n(z)}{U_n'(z)},\quad
n\in\mathbb{N},x\in\mathbb{R}.
\end{eqnarray}
We now restate  Assumptions \ref{elastic} and \ref{compact} for
the family $(\tilde{U}_n)_n$.
\begin{assumption}\label{elastictilde} For each $[\lambda_0,\lambda_1]\subset (0,\infty)$ there exist constants
$C_1,C_2,C_3$ such that for all $n$ and for all $y>0$,
\begin{equation}\label{G}
\tilde{V}_n(\lambda y)\leq C_1\tilde{V}_n(y)+C_2y+C_3.
\end{equation}
holds for each $\lambda\in [\lambda_0,\lambda_1]$.
\end{assumption}
\begin{assumption}\label{compacttilde}
There exists $Q_0\in\mathcal{M}$ such that the sequence $\tilde{V}_n(dQ_0/dP)$ is uniformly
integrable (with respect to $P$).
\end{assumption}
\begin{remark}
We think that Assumptions \ref{elastictilde} and \ref{compacttilde} should be considered as technical
assumptions to be checked on $\tilde{U}_n$. We will see below that they hold true in our examples.
\end{remark}
\begin{theorem}\label{main2} Assume that Assumptions \ref{a1}, \ref{elastictilde}
and \ref{compacttilde} hold. Then
for each bounded random variable $G$ the
indifference prices $p_n(z,G)$ tend to $\pi(G)$ as $n\to\infty$.
\end{theorem}
\begin{proof}
Obviously, $\tilde{U}_n(z)=0$ and $\tilde{U}_n'(z)=1$, thus Assumption
\ref{a2} is satisfied for the sequence $\tilde{U}_n$ at $x_0=z$.
Let $\tilde{u}_n$ be defined analogously as in (\ref{max}), with $\tilde{U}_n$ replacing $U_n$.
Then since the $\tilde{U}_n$ are affine transforms of the $U_n$, we
may alternatively write
\[
p_n(z,G)=\inf\{ p\in\mathbb{R}: \tilde{u}_n(z+p,G)\geq
\tilde{u}_n(z,0)\}.
\]
It is obvious that if Assumption \ref{a1} hold for $U_n$ then it holds also for $\tilde{U}_n$,
so Theorem \ref{main} applied for $\tilde{U}_n$ and $x_0=z$ allow us to
conclude.
\end{proof}

\begin{example}(The exponential case) Let
\[
U_n(x)=\frac{1-\exp\{-\alpha_n x\}}{\alpha_n},
\]
with some $\alpha_n>0$ tending to $\infty$ as $n\to\infty$.
It is straightforward that $U_n(0)=0$, $U'_n(0)=1$, $r_n(x)= \alpha_n$ and thus
$U_n$ satisfy Assumption \ref{a1} and \ref{a2} for $x_0=0$.
Moreover, calculation gives
${V}_n(y)=(1/\alpha_n)[y\ln y+1-y]$ showing that Assumption
\ref{compact} holds provided that there is $Q_0\sim P$ with $E_{Q_0}
\ln(dQ_0/dP)<\infty$, i.e. whenever a finite-entropy martingale
measure exists. It is enough to show \eqref{GRRR} for one function,
namely $V(y)=y\ln y$. This is trivial since $\lambda y\ln (\lambda
y)\leq \lambda y\ln y+ \lambda \vert\ln\lambda\vert y$ for
$\lambda,y>0$.

So Theorem \ref{main} applies and the respective reservation prices $p_n(0,G)$
converge to $\pi(G)$.

In fact, one can check that Theorem \ref{main2} applies for each $z\in\mathbb{R}$.
We have thus retrieved the
result in \cite{sixauthor}. (They make a weaker moment assumption on $G$; this extension
follows by our method, too.)
\begin{eqnarray*}
\end{eqnarray*}
\end{example}

\begin{example}(The power case)
\[
U_n(x)=-\frac{1}{\alpha_n}[ (x+1)^{-\alpha_n}-1]
1_{\{x>0\}} -\frac{1}{\beta_n}[ (1-x)^{\beta_n}-1]1_{\{x\leq 0\}}
\]
with $\alpha_n>0,\beta_n>1$ tending to $\infty$. These functions are continuously
differentiable with strictly monotone derivatives, hence they are strictly concave.
They are also twice continuously differentiable on $\mathbb{R}\setminus\{0\}$ and satisfy
\eqref{r} on $\mathbb{R}\setminus\{0\}$. Take now $\beta_n:=\alpha_n+2$
then $U_n''$ exists and \eqref{r} holds in $0$, too.
\begin{eqnarray*}
U'_n(x) & = & (x+1)^{-(\alpha_n+1)}
1_{\{x>0\}} - (1-x)^{\alpha_n+1}1_{\{x < 0 \}} \\
U''_n(x) & = & -(\alpha_n+1)(x+1)^{-(\alpha_n+2)}
1_{\{x>0\}} - (\alpha_n+1)(1-x)^{\alpha_n} 1_{\{x <0 \}}\}
\end{eqnarray*}
Thus $U_n(0)=0$, $U'_n(0)=1$, $r_n(x)= (\alpha_n+1)(\frac{1_{\{x>0\}}}{x+1} + \frac{1_{\{x<0\}}}{1-x})
$ and $U_n$ satisfy Assumptions \ref{a1} and \ref{a2} for $x_0=0$.

It remains to check \eqref{GRRR} and Assumption \ref{compact}. We have, for $y\in (0,\infty)$,
\[
{V}_n(y)=\left[\frac{1}{\alpha_n+2}-y+\left(1-\frac{1}{\alpha_n+2}\right)y^{\frac{\alpha_n+2}{\alpha_n+1}}\right]1_{\{
y>1\}}
+\left[\frac{1}{\alpha_n}+y-\left(1+\frac{1}{\alpha_n}\right)y^{\frac{\alpha_n}{\alpha_n+1}}\right]1_{\{
y\leq 1\}}.
\]

Thus if there exists $Q_0\in\mathcal{M}$ such that $dQ_0/dP\in
L^{1+\varepsilon}$ for some $\varepsilon>0$ then Assumption
\ref{compact} will be satisfied for $n$ large enough since
$(\alpha_n+2)/(\alpha_n+1)$ converges to $1$. Finally, \eqref{GRRR} is
trivially satisfied by the scaling properties of the power
functions. We may conclude by Theorem \ref{main} that the
reservation prices $p_n(0,G)$ tend to $\pi(G)$ as $n\to\infty$.

Tedious calculations show that we get the same conclusion for any
initial capital $z\in\mathbb{R}$ and we may apply Theorem
\ref{main2}.
\end{example}

\end{document}